\documentclass[11pt,a4paper]{article}
\usepackage[T1]{fontenc}
\usepackage{amsfonts,amssymb,amsmath,textcomp,amsthm}
\usepackage{indentfirst}                       
\usepackage{ae}                                
\usepackage{geometry}
\usepackage{graphicx}
\usepackage{hyperref}
\usepackage{float}
\usepackage[dvipsnames]{xcolor}
\usepackage{fancyhdr}
\newtheorem{teo}{Theorem}[section]

\newtheorem{lemma}[teo]{Lemma}
\newtheorem{cor}[teo]{Corollary}
\newtheorem{prop}[teo]{Proposition}
\theoremstyle{definition}

\newtheorem{exem}[teo]{Example}

\newtheorem{obs}[teo]{Remark}

\def\blfootnote{\xdef\@thefnmark{}\@footnotetext} 

\date{}

\author{GILIARD SOUZA DOS ANJOS\footnote{This study was financed in part by the Coordena\c c\~ao de Aperfei\c coamento de Pessoal de N\'ivel Superior - Brasil (CAPES) - Finance Code 001}\\[0.2cm]Instituto de Matemática e Estatística - Universidade de São Paulo
\\Rua do Matão, 1010, 05508-090, São Paulo - SP, Brazil\\giliards@ime.usp.br}

\title{\Large \textbf{HALF-AUTOMORPHISM GROUP OF CHEIN LOOPS}}

\begin{document}

\maketitle

\begin{abstract} 
\noindent{}A bijection $f$ of a loop $L$ is a \emph{half-automorphism} if $f(xy)\in \{f(x)f(y),f(y)f(x)\}$, for any $x,y\in L$. A half-automorphism is \emph{nontrivial} when it is neither an automorphism nor an anti-automorphism. A \emph{Chein loop} $L=G\cup Gu$ is a Moufang loop constructed from a group $G$ and an element $u$ of order $2$ outside of $G$. In this paper, the half-automorphism group of finite Chein loops is described.

\end{abstract}

\noindent{}{\it Keywords}: Chein loop, half-isomorphism, half-automorphism, half-automorphism group.

\section{Introduction}
A \emph{loop} $( L,*) $ consists of a nonempty set $L$ with a binary operation $*$ on $L$ such that for each $ a, b \in L $ the equations $ a * x  =  b $ and $ y * a =  b$ have unique solutions, and there is an identity element $ 1 \in L$ such that $ 1 * x = x = x *  1$, for any $ x \in L $. 

A \emph{Moufang loop} is a loop that satisfies the identity $ x(y(xz))=((xy)x)z$. Every Moufang loop is \emph{diassociative}, that is, any two of its elements generate a group. The fundamental definitions and facts from loop theory and Moufang loops are found in \cite{B71,P90}.

Let $G$ be a group and $u$ be an indeterminate. On $M(G,2) = G \cup Gu$, define the operation $*$ by

\vspace{-0.5cm}
\begin{equation}
g*h = gh,\hspace{0.5cm} g*(hu) = (hg)u,\hspace{0.5cm} (gu)*h = (gh^{-1})u,\hspace{0.5cm} (gu)*(hu) = h^{-1}g\hspace{0.5cm}
\end{equation}

\noindent{}for any $g,h\in G$. Then, $(M(G,2),*)$ is a Moufang loop. Loops of this type are known as \emph{Chein loops} \cite{C74}. Note that $M(G,2)$ is associative if and only if $G$ is abelian.


Let $(L,*)$ and $(L',\cdot)$ be loops. A bijection $f: L \rightarrow L'$ is called a \emph{half-isomorphism} if  $f(x*y)\in\{f(x)\cdot f(y),f(y)\cdot   f(x)\}$, for any $ x, y \in L$. A half-isomorphism of a loop onto itself is a \emph{half-automorphism}. We say that a half-isomorphism is \emph{trivial} when it is either an isomorphism or an anti-isomorphism.

In 1957, Scott \cite{Sco57} showed that every half-isomorphism on groups is trivial. Recently, similar versions of this result were obtained for more structured classes of nonassociative loops. Kinyon, Stuhl and Vojt\v echovsk\' y \cite{KSV16} generalized Scott's result to a more general subclass of Moufang loops. Gagola and Giuliani \cite{GG19} proved that every half-automorphism of a finite simple Moufang loop is trivial. In \cite{GA192}, the authors showed that in a subclass of automorphic loops of odd order every half-automorphism is trivial. An \emph{automorphic loop} is a loop in which every inner mapping is an automorphism \cite{BP56}.

Gagola and Giuliani also investigated loops that have nontrivial half-automorphisms. In \cite{GG13}, these authors established conditions for the existence of nontrivial half-automorphisms for Chein loops. A similar work was made for certain automorphic loops of even order in \cite{GA19}, where the authors also described the half-automorphism group of these loops. This work continues the investigation of nontrivial half-automorphisms in Chein loops initiated by Gagola and Giuliani. The aim is to describe the half-automorphism group of Chein loops.

This paper is organized as follows: Section~\ref{sec2} presents the definitions and basic results about half-automorphisms on Chein loops. In section~\ref{sec3}, some properties of an important subgroup of the half-automorphism group of Chein loops are obtained. In section~\ref{sec4}, the half-automorphism group of finite Chein loops is described.

\section{Preliminaries}
\label{sec2}

Here, the required definitions and basic results on half-automorphisms of Chein loops are stated.

First, a description of the automorphism group of Chein loops made by Grishkov, Zavarnitsine and Giuliani \cite{GZG05} is presented. Let $G$ be a group and $L = M(G,2)$. Denote the automorphism group of $G$ by $Aut(G)$. For $\phi \in Aut(G)$ and $t\in G$, define $a_\phi,d_t: L \rightarrow L$ by:

\begin{center}
$a_\phi(gu^i) = (\phi(g)) u^i, \,\,\, i \in \{ 0,1\},$ \hspace{1cm} $d_t(gu^i) = \left\{\begin{array}{rl}
g, & i = 0, \\
(tg)u, & i = 1.
\end{array}\right.$
\end{center}

Let $H$ and $H'$ be groups and $\sigma: H \rightarrow Aut(H')$ be a homomorphism. The set $H\times H'$ with the operation $(x,y)\cdot (x',y') = (xx',y\sigma_x(y'))$, where $\sigma_x = \sigma(x)$, is a semidirect product between $H$ and $H'$, and it is denoted by $H \stackrel{\sigma}{\ltimes} H'$. If $\pi: Aut(G) \rightarrow Aut(G)$ is the identity mapping, we can construct the semidirect product $Aut(H)\stackrel{\pi}{\ltimes} H$, which is called the \textit{holomorph of $G$} and denoted by $Hol(G)$.

\begin{prop} 
\label{prop1}(\cite[Lemma 1]{GZG05}) Let $A = \{a_\phi \,\,|\,\, \phi \in Aut(G)\}$ and $D = \{d_t \,\, |\,\, t\in G\}$. Then:

\noindent{}a) $A,D \leq Aut(L)$, and $A\cap D = \{I\}$, where $I$ is the identity mapping of $L$,
\\
b) $a_\phi a_{\phi'} = a_{\phi \phi'} $, $d_td_{t'} = d_{tt'}$, $(a_\phi)^{-1} = a_{\phi^{-1}}$, $(d_t)^{-1} = d_{t^{-1}}$, and $ a_\phi d_t  =  d_{\phi(t)} a_\phi$,
\\
c) $A \cong Aut(G)$, $D \cong G$, and $AD \cong Hol(G)$.
\end{prop}

If there exists an abelian subgroup $G_0\leq G$ of index $2$ such that $G = G_0 \cup G_0v$, where $v\not \in G_0$, $v^2 = 1$ and $vgv = g^{-1}$, for all $g \in G_0$, then $G$ is a \textit{generalized dihedral group}. In this case, $L = M(G,2) = G_0K$, where $K = \{1,u,v,w=uv=vu\}$ is isomorphic to the Klein group. For $\alpha \in Aut(K) \cong S_3$, we can define an automorphism $s_\alpha$ of $L$ by $s_\alpha (gx) = g\alpha(x)$, for any $ x\in K, g\in G_0$. Define $S = \{s_\alpha \,\, | \alpha \in Aut(K)\}$. Thus, $S \cong S_3$.

\begin{teo} 
\label{teo1}(\cite[Theorem 1]{GZG05}) Let $G$ be a group. If $G$ is not a generalized dihedral group, then $Aut(M(G,2)) \cong Hol(G)$. If $G = G_0 \cup G_0v$ is a generalized dihedral group and $G_0$ is not a group of exponent $2$, then $Aut(M(G,2)) = ADS$.
\end{teo}

Note that if $G = G_0 \cup G_0v$ is a generalized dihedral group and $G_0$ is not a group of exponent $2$, then $AD \cap S = \{I,d_v\}$. Thus, $|Aut(M(G,2))| = 3\,|Hol(G)| = 3\,|Aut(G)|\,|G|$.

Denote the sets of the half-automorphisms and trivial half-automorphisms of the loop $L$ by $\mathit{Half}(L)$ and $\mathit{Half_T}(L)$, respectively.

In any loop $Q$, it is easy to see that the composition of two half-automorphisms is also a half-automorphism, the identity mapping of $Q$ is a half-automorphism and the inverse mapping of a trivial half-automorphism is also a trivial half-automorphism. Then, $\mathit{Half_T}(Q)$ is always a group, and it is called the \emph{trivial half-automorphism group of $Q$}.

Let $f$ be a half-automorphism of a diassociative loop $M$. In \cite[Lemma 2.1]{KSV16}, it was proved that if $x,y\in M$ commute, then $f(x)$ and $f(y)$ commute. Thus, $f^{-1}$ is a half-automorphism by \cite[Theorem 2.5]{GA192}. Hence, $\mathit{Half}(M)$ is a group, and it is called the \emph{half-automorphism group of $M$}.

We conclude that $\mathit{Half}(L)$ and $\mathit{Half_T}(L)$ are groups for every Chein loop $L$.

In the following, some results about half-automorphisms in Chein loops obtained by Gagola and Giuliani in \cite{GG13} are presented.

Denote the order of the element $x$ of a group $G$ by $o(x)$ and the center of $G$ by $\mathcal{Z}(G)$.

\begin{teo} 
\label{teo2}(\cite[Theorem 1]{GG13}) Let $G$ be a finite nonabelian group and $L = M(G,2)$. If $L$ has a nontrivial half-automorphism, then there exists an element $x\in G$ such that:

\noindent{}a) $o(x) = 4$,
\\
b) $x^2\in \mathcal{Z}(G)$,
\\
c) $x^{-1}gx = g^{-1}$ for any element $g\in G$ such that $o(g)\not \in\{2,4\}$.
\end{teo}

Let $G$ be a group and $L = M(G,2)$. Define the set 

\begin{center}
$H = \{\varphi \in \mathit{Half}(L)\,\,|\,\,\varphi(u) = u \textrm{ and } \varphi(g)=g, \forall \,\,g\in G  \}$. 
\end{center}

Then, $H$ is a subgroup of $\mathit{Half}(L)$. Analyzing the proof of \cite[Lemma 5]{GG13}, this lemma can be rewritten as follows.

\begin{prop}
\label{prop3} (\cite[Lemma $5$]{GG13}) $\mathit{Half}(L) = \mathit{Half}_T(L)H$. In particular, if $L$ has a nontrivial half-automorphism, then there exists a nontrivial half-automorphism in $H$.
\end{prop}

In the proof of Theorem \ref{teo2}, the authors in \cite{GG13} proved the following result about the elements of $H$.

\begin{prop}
\label{prop4} (\cite{GG13}) Let $\varphi \in H$ and $y\in G$. If $o(y)\not = 4$, then $\varphi(yu) = yu$.
\end{prop}

\section{The properties of the elements of H}
\label{sec3}

In this section, $G$ is considered as a finite nonabelian group such that $L = M(G,2)$ has a nontrivial half-automorphism.

Denote the identity mapping of $L$ by $I$. By Proposition \ref{prop3}, there exists $\varphi \in H$ such that $\varphi \not = I$.

\begin{prop}
\label{prop5} Let $\varphi \in H\setminus \{I\}$. Then $\varphi$ is a nontrivial half-automorphism.
\end{prop}
\begin{proof}
Since $G$ is nonabelian and $\varphi(g) = g$, for all $g\in G$, the mapping $\varphi$ is not an anti-automorphism. If $\varphi$ is an automorphism, then $\varphi(gu) = \varphi(g)*\varphi(u) = gu$, for all $g \in G$, and hence $\varphi = I$.
\end{proof}

Let $\varphi \in H$. Since $\varphi$ is a half-automorphism, we have $\varphi(gu) \in \{\varphi(g)*\varphi(u),\varphi(u)*\varphi(g)\}$, for all $g\in G$. Then,

\begin{equation}
\label{eq1}
\varphi(gu) \in \{gu,g^{-1}u\}, \textrm{ for all } g\in G.
\end{equation}

Since $\varphi$ is a bijection, it follows that $\varphi^2 = I$. Thus, $H$ is an abelian group of exponent $2$. Denote by $C_2^n$ the direct product of $n$ cyclic groups of order $2$. Hence, the following result is at hand.

\begin{prop}
\label{prop52} $H \cong C_2^r$, for some $r \geq 1$.
\end{prop}

Let $\varphi\in H\setminus \{I\}$. Define the set $\gamma_\varphi =\{ x\in G \,\,|\,\, \varphi(xu) = x^{-1}u, o(x) = 4\}$. By Theorem \ref{teo2}, $\gamma_\varphi \not = \emptyset$.

\begin{prop}
\label{prop6} Let $\varphi\in H\setminus \{I\}$. Then, for all $g,h \in G$ and $x\in \gamma_\varphi$:

\noindent{}(a) If $g\not \in \gamma_\varphi$, then $\varphi(gu) = gu$, $g^{-1} = x^{-1}gx$ and $o(xg) = 4$,
\\
(b) $x^{-1} \in \gamma_\varphi$,
\\
(c) $x^2 \in \mathcal{Z}(G)$,
\\
(d) If $o(g)\not = 4$, then $\gamma_\varphi g = g\gamma_\varphi = \gamma_\varphi$,
\\
(e) If $g,h\not \in \gamma_\varphi$ and $hg \in \gamma_\varphi$, then $o(g) = 4$.
\end{prop}
\begin{proof}

(a) Let $g \in G \setminus \gamma_\varphi$. By Proposition \ref{prop4} and \eqref{eq1}, $\varphi(gu) = gu$. Then,

\begin{center}
$g^{-1}x = \varphi((xu)*(gu)) \in \{\varphi(xu)*\varphi(gu),\varphi(gu)*\varphi(xu))\} = \{g^{-1}x^{-1},xg\}$.
\end{center}

Since $x\not = x^{-1}$, we have $g^{-1}x = xg$. Thus, $g^{-1} = xgx^{-1}$, and hence $g^{-1} = x^{-1}gx$. Furthermore, $xgxg = xgg^{-1}x = x^2$, and therefore $o(xg) = 4$.\\

\noindent{}(b) It follows from \eqref{eq1}.\\

\noindent{}(c) Let $g \in G \setminus \gamma_\varphi$. By (a), $xgx^{-1} = g^{-1} = x^{-1}gx$. Hence, $x^2g = gx^2$.

Now, let $y\in \gamma_\varphi$. Then,

\begin{center}
$y^{-1}x = \varphi((xu)*(yu)) \in \{\varphi(xu)*\varphi(yu),\varphi(yu)*\varphi(xu))\} = \{yx^{-1},xy^{-1}\}$.
\end{center}

Thus, either $y^{-1}x = yx^{-1}$, or $y^{-1}x = xy^{-1}$. In both cases, we can conclude that $x^2y = yx^2$. Therefore, $x^2 \in \mathcal{Z}(G)$.\\

\noindent{}(d) 
Since $x\in \gamma_\varphi$, we have

\begin{center}
$\varphi(g)*\varphi(xu) = g*(x^{-1}u) = (x^{-1}g)u$,\hspace{0.5cm} $\varphi(xu)*\varphi(g) = (x^{-1}u)*g = (x^{-1}g^{-1})u$
\\
and $\varphi(g*(xu)) = \varphi((xg)u) = (xg)^{\epsilon}u$,
\end{center}

\noindent{}where $\epsilon \in \{-1,1\}$. Thus, $(xg)^{\epsilon} \in \{x^{-1}g,x^{-1}g^{-1}\}$. If $\epsilon = 1$, then $xg = x^{-1}g^{-1}$ since $x\not = x^{-1}$, and so $g^2 = x^{-2} = x^2$. Hence, $o(g) = 4$, which is a contradiction. Therefore, $\epsilon = -1$. By (a), $g^{-1}x = xg\in \gamma_\varphi$.

Since $o(g^{-1})\not = 4$, we conclude that $gx\in \gamma_\varphi$. The rest of the claim follows from the fact that $G$ is finite.\\

\noindent{}(e) Let $g,h \in G \setminus\gamma_\varphi$ be such that $hg \in \gamma_\varphi$. Then,

\begin{center}
$\varphi(g)*\varphi(hu) = g*(hu) = (hg)u$,\hspace{0.5cm} $\varphi(hu)*\varphi(g) = (hu)*g = (hg^{-1})u$
\\
and $\varphi(g*(hu)) = \varphi((hg)u) = (hg)^{-1}u$.
\end{center}

Thus, $(hg)^{-1} \in \{hg,hg^{-1}\}$. Since $o(hg) = 4$, we have $g^{-1}h^{-1} = hg^{-1}$, and then $gh = h^{-1}g$. Hence, $hghg = h(h^{-1}g)g = g^2$ and we conclude that $o(g) = 4$.

\end{proof}

\begin{cor}
\label{cor1} There exists $y\in G$ such that:

\noindent{}(a) $o(y) = 4$ and 
\\
(b) if $y = gh$, for some $g,h\in G$, then either $o(g) = 4$, or $o(h) = 4$.
\end{cor}
\begin{proof}
Fix $\varphi \in H \setminus \{I\}$. Then, $\varphi$ is a nontrivial half-automorphism and $\gamma_\varphi \not = \emptyset$. Let $y \in \gamma_\varphi$. Suppose there exist $g,h\in G$ such that $y = gh$, $o(g)\not = 4$ and $o(h)\not = 4$. Thus, $g,h\not \in \gamma_\varphi$. By the item (e) of Proposition \ref{prop6}, $o(h) = 4$, which is a contradiction.
\end{proof}

Suppose that $G = G_0\cup G_0v$ is a generalized dihedral group. Note that $o(gv) =2$, for any $g\in G_0$. Let $z \in G$ be of order $4$. Then, $z\in G_0$ and $z = (z^{-1}v)(z^2v)$. Therefore, there is no element of $G$ satisfying the condition of Corollary \ref{cor1}. Hence, the following result is at hand.

\begin{cor}
\label{cor2} If $G$ is a finite generalized dihedral group, then every half-automorphism of $M(2,G)$ is trivial.
\end{cor}

By Theorem \ref{teo2}, $|\mathcal{Z}(G)|\geq 2$. In the following result, this subgroup is described.

\begin{teo}
\label{prop7} $\mathcal{Z}(G) \cong C_2^m$, for some $m\geq 1$.
\end{teo}
\begin{proof} Fix $\varphi \in H\setminus \{I\}$. First, we will prove that $\mathcal{Z}(G) \cap \gamma_\varphi = \emptyset$.

Suppose there exists $z\in \mathcal{Z}(G) \cap \gamma_\varphi$. Let $g\in G \setminus \gamma_\varphi$. By the item (a) of Proposition \ref{prop6}, $g^{-1} = z^{-1}gz = g$. By \eqref{eq1}, $\varphi(gu) = g^{-1}u,$  for all $g \in G$. Consequently, for  $g,h \in G$, we have

\begin{center}
$\varphi(gu*h) = (hg^{-1})u$, $\varphi(gu)*\varphi(h) = (g^{-1}h^{-1})u$ and $\varphi(h)*\varphi(gu) = (g^{-1}h)u$.
\end{center}

Thus, either $hg = gh$, or $h^{-1} = ghg^{-1}$, for all $g,h \in G$. Since $G$ is nonabelian, there exist $x,y\in G$ such that $xy\not = yx$, and then $y^{-1} = xyx^{-1}$. Since $x(yz) \not = yxz = (yz)x$, we have $(yz)^{-1} = xyzx^{-1} = y^{-1}z$, and hence $y^{-1} = z^2y^{-1}$. Therefore, $o(z)= 2$, a contradiction with $z\in \gamma_\varphi$. We conclude that $\mathcal{Z}(G) \cap \gamma_\varphi = \emptyset$.


Now, let $z\in \mathcal{Z}(G)$ and $x\in \gamma_\varphi$. By the item (a) of Proposition \ref{prop6}, $z^{-1} = x^{-1}zx = z$. Thus, $\mathcal{Z}(G)$ has exponent $2$, and hence there exists $m\geq 1$ such that $\mathcal{Z}(G) \cong C_2^m$.
\end{proof}

\section{The half-automorphism group}
\label{sec4}

We say that a loop $Q$ has the \emph{antiautomorphic inverse property} (AAIP) if it has two-sided inverses and satisfies the identity $(xy)^{-1} = y^{-1}x^{-1}$. In this type of loop, the \emph{inversion mapping} $J: Q \rightarrow Q$, defined by $J(x) = x^{-1}$, is an anti-automorphism. Furthermore, it is easy to see that $J \in \mathcal{Z}(\mathit{Half_T}(Q))$.

\begin{prop}
\label{prop8} Let $Q$ be a nonabelian loop with the (AAIP). Then, 

\begin{center}
$\mathit{Half}_T(Q) = \left\langle J\right\rangle Aut(Q) \cong C_2 \times Aut(Q)$.
\end{center}
\end{prop}
\begin{proof} Since $J \in \mathcal{Z}(\mathit{Half}_T(Q))$, $J$ has order $2$ and $J\not \in Aut(Q)$, it follows that $\left\langle J\right\rangle Aut(Q) \cong C_2 \times Aut(Q)$. 

Denote the set of the anti-automorphisms of $Q$ by $Ant(Q)$. Since $Q$ is nonabelian, we have $\mathit{Half}_T(Q) = Aut(Q) \cup Ant(Q)$, where $Aut(Q) \cap Ant(Q) = \emptyset$. It therefore suffices to verify that $Ant(Q) = J \,Aut(Q)$.

Let $\varphi \in Aut(Q)$. Then, $J\varphi(xy) = J\varphi(y) J\varphi(x)$, for all $x,y \in G$, and so $J \,Aut(Q) \subset Ant(Q)$. On the other hand, let $\varphi \in Ant(Q)$. Then, $J\varphi(xy) = J\varphi(x) J\varphi(y)$,  for all $x,y \in G$. Thus, $J\varphi \in Aut(Q)$, and hence $\varphi \in J \,Aut(Q)$. 
\end{proof}

\begin{obs}
\label{obs1}
If $Q$ is an abelian loop, then $\mathit{Half}(Q) = \mathit{Half}_T(Q) = Aut(Q)$.
\end{obs}

Every Moufang loop has the (AAIP). Thus, Theorem \ref{teo1}, Proposition \ref{prop8} and Remark \ref{obs1} give us a description of the trivial half-automorphism group of Chein loops.

From now on, we will suppose that the Chein loop $L = M(G,2)$ is finite and has a nontrivial half-automorphism. Then, $G$ is a finite nonabelian group that is not a generalized dihedral group.

By Theorem \ref{teo1} and Propositions \ref{prop3} and \ref{prop8}, $\mathit{Half}(L) = \left\langle J\right\rangle ADH$, where $A$ and $D$ were defined in section \ref{sec2}, and $J$ is the inversion mapping of $L$.

\begin{lemma}
\label{lema1} Let $\varphi \in H$, $\phi \in Aut(G)$, and $t\in G$. Then:

\noindent{}(a) $(a_\phi)^{-1} \varphi a_\phi\in H$,
\\
(b) If $t\in G - \gamma_\varphi$, then $(d_t)^{-1}\varphi d_t \in H$. If $t\in \gamma_\varphi$, then $d_t \varphi d_t \in H$.
\end{lemma}
\begin{proof} For $g\in G$, we have that $(a_\phi)^{-1}\varphi a_\phi (g)= \phi^{-1}\phi (g) = g$ and $(d_t)^{-1}\varphi d_t (g) = g = d_t \varphi d_t (g)$. 

Consider $u\in L$. Thus, $(a_\phi)^{-1} \varphi a_\phi(u) =  (a_\phi)^{-1} \varphi(u)= a_{\phi^{-1}}(u)  = u$. If $t\in  G \setminus \gamma_\varphi$, then $\varphi(tu) = tu$ by Proposition \ref{prop6}, and so $ (d_t)^{-1} \varphi d_t (u)= d_{t^{-1}}(tu) = u$. If $t\in  \gamma_\varphi$, then $\varphi(tu) = t^{-1}u$, and hence $d_t \varphi d_t (u)= d_t(t^{-1}u) = u$.
\end{proof}

By Proposition \ref{prop5}, $H\cap AD = \{I\}$. Since $H,AD \leq \mathit{Half}(L)$, it follows that $DH, ADH \leq \mathit{Half}(L)$ by Lemma \ref{lema1}. In \cite[Lemma 2.1]{KSV16}, it was proved that $fJ = Jf$, for all $f\in \mathit{Half}(L)$. Thus, $J\in \mathcal{Z}(\mathit{Half}(L))$. Furthermore, it is clear that $J\not \in ADH$. Consequently, we have

\begin{equation}
\label{eq2}
\mathit{Half}(L) \cong C_2 \times ADH.
\end{equation}

Now, we will describe the group $ADH$. For $a_\phi\in A$, define $\sigma_\phi: DH \rightarrow DH$ by $\sigma_\phi(d_t \varphi) = a_\phi d_t \varphi (a_\phi)^{-1} = d_{\phi(t)}a_\phi\varphi (a_\phi)^{-1}$, where $\varphi \in H$ and $d_t\in D$. By Lemma \ref{lema1}, $\sigma_\phi$ is well-defined, so it is the restriction of an inner automorphism of $ADH$ to $DH$. Therefore, $\sigma_\phi\in Aut(DH)$ and $\sigma_\phi\sigma_{\phi'} = \sigma_{\phi\phi'}$, for any $a_\phi,a_{\phi'}\in A$.

Define $\sigma: A \rightarrow Aut(DH)$ by $\sigma(a_\phi) = \sigma_\phi$. Then, the mapping $\sigma$ is a homomorphism. Thus, we can construct the semidirect product $A \stackrel{\sigma}{\ltimes} DH$, where the operation is given by

\begin{center}
$(a_\phi,d_t\varphi)\stackrel{\sigma}{\cdot} (a_{\phi'},d_{t'}\varphi') = (a_{\phi\phi'},d_t\varphi d_{\phi(t')}a_\phi\varphi' (a_\phi)^{-1})$.
\end{center}

In $ADH$, we have that $d_t\varphi a_\phi \, d_{t'}\varphi'a_{\phi}'= d_t\varphi d_{\phi(t')}a_\phi\varphi' (a_\phi)^{-1}a_{\phi\phi'}$, and then the mapping $\psi: ADH \rightarrow A \stackrel{\sigma}{\ltimes} DH$, defined by $\psi(d_t\varphi a_\phi) = (a_\phi,d_t\varphi)$, is an isomorphism. Therefore, we established the following result.

\begin{prop}
\label{prop9} $ADH \cong A \stackrel{\sigma}{\ltimes} DH \cong Aut(G) \stackrel{\sigma}{\ltimes} DH$.
\end{prop}

Now, we will describe the group $DH$.

\begin{prop}
\label{prop10} $Z = \{d_z \,\,|\,\, z\in \mathcal{Z}(G)\} \leq \mathcal{Z}(DH)$.
\end{prop}
\begin{proof}
Let $z\in \mathcal{Z}(G)$ and $d_t\in D$. Then, $d_zd_t = d_{zt}= d_{tz} = d_td_z$.

Now, let $\varphi \in H$. Since $o(z)\in \{1,2\}$, $z \gamma_\varphi = \gamma_\varphi$ by the item (d) of Proposition \ref{prop6}. If $g\in \gamma_\varphi$, then $zg\in \gamma_\varphi$, and so $d_z\varphi (gu) = (zg^{-1})u = (g^{-1}z^{-1})u = \varphi d_z(gu)$. If $g\in G \setminus \gamma_\varphi$, then $zg\in G \setminus \gamma_\varphi$, and hence $d_z\varphi (gu) = (zg)u = \varphi d_z(gu)$.

Therefore, $Z \subset \mathcal{Z}(DH)$. Since $Z \cong \mathcal{Z}(G)$, we conclude that $Z \leq \mathcal{Z}(DH)$.
\end{proof}

Since $Z \subset D$, it follows that $Z\cap H = \{I\}$. By Proposition \ref{prop10}, $ZH \leq DH$.

For $x\in G$, the mapping $I_x: G \rightarrow G$, defined by $I_x(g) = gxg^{-1}$, is an inner automorphism of $G$. Let $\mathcal{I}(G) =\{I_x\,\,|\,\,x\in G\}$ be the inner automorphism group of $G$. It is well known that $G/\mathcal{Z}(G) \cong \mathcal{I}(G)$.

\begin{prop}
\label{prop11} We have:

\noindent{}a) $ZH \cong C_2^n$, where $|\mathcal{Z}(G)| = 2^m$, $|H| = 2^r$, and $n = m + r$,
\\
b) $ZH \triangleleft DH$,
\\
c) $DH / ZH \cong \mathcal{I}(G)$.
\end{prop}
\begin{proof}
a) Since $Z \leq \mathcal{Z}(DH)$ and $Z\cap H = \{I\}$, we have $ZH \cong Z\times H$. The result follows from Proposition \ref{prop52} and Theorem \ref{prop7}.\\

\noindent{}b) Let $z\in \mathcal{Z}(G)$, $t\in G$, and $\varphi,\varphi'\in H$. Then,

\begin{center}
$\varphi' d_z \varphi (\varphi')^{-1} = d_z \varphi' \varphi \varphi' = d_z\varphi$ \hspace{0.1cm} and \hspace{0.1cm} $d_t d_z \varphi d_{t^{-1}}  = d_z d_t \varphi d_{t^{-1}} = d_{zt^2} d_{t^{-1}} \varphi d_{t^{-1}}$.
\end{center}

It is clear that $\varphi' d_z \varphi (\varphi')^{-1} \in ZH$. If $t\in \gamma_\varphi$, then $t^2 \in \mathcal{Z}(G)$ and $d_{t^{-1}} \varphi d_{t^{-1}} \in H$ by Proposition \ref{prop6} and Lemma \ref{lema1}. If $t\not \in \gamma_\varphi$, then $d_t \varphi d_{t^{-1}} \in H$ by Lemma \ref{lema1}. Thus, $d_t d_z \varphi d_{t^{-1}}  \in ZH$, for all $t\in G$. Therefore, $ZH \triangleleft DH$.\\

\noindent{}c) Define $\psi: DH \rightarrow \mathcal{I}(G)$ by $\psi(d_t\varphi) = I_t$. It is clear that $\psi$ is surjective. 

Let $t,t'\in G$ and $\varphi,\varphi'\in H$. If $t'\in \gamma_\varphi$, then $t'^2 \in \mathcal{Z}(G)$, and hence $I_{t'}(g) = t'gt'^{-1} = t'^{-1}t'^2gt'^{-1} = t'^{-1}gt'^2t'^{-1} = I_{t'^{-1}}(g)$, for all $g\in G$. Since $d_{t'} \varphi d_{t'} \in H$, we have 

\begin{center}
$\psi(d_t\varphi d_{t'}\varphi') = \psi(d_{tt'^{-1}}d_{t'}\varphi d_{t'}\varphi') = I_{tt'^{-1}} = I_tI_{t'^{-1}} = I_tI_{t'} = \psi(d_t\varphi)\psi(d_{t'}\varphi')$.
\end{center}

If $t'\not \in \gamma_\varphi$, then $d_{t'^{-1}} \varphi d_{t'} \in H$. Thus, 

\begin{center}
$\psi(d_t\varphi d_{t'}\varphi') = \psi(d_{tt'}d_{t'^{-1}}\varphi d_{t'}\varphi') = I_{tt'} = I_tI_{t'} = \psi(d_t\varphi)\psi(d_{t'}\varphi')$.
\end{center}

Hence, $\psi$ is a homomorphism.

Suppose that $\psi(d_t\varphi) = I$. Then, $g = I_t(g) = tgt^{-1}$, for all $g\in G$. Thus, $t \in \mathcal{Z}(G)$, and hence $Ker(\psi) = ZH$. Therefore, $DH / ZH \cong \mathcal{I}(G)$.
\end{proof}

For $t\in G$, define $\rho_t: ZH \rightarrow ZH$ by $\rho_t(d_z\varphi) = d_t d_z\varphi d_{t^{-1}}$, where $d_z\in Z$ and $\varphi \in H$. By the item (b) of Proposition \ref{prop11}, $\rho_t$ is well-defined, and then it is the restriction of an inner automorphism of $DH$ to $ZH$. Thus, $\rho_t \in Aut(ZH)$ and $\rho_t\rho_{t'} = \rho_{tt'}$, for any $t,t'\in G$.

Define $\rho: G \rightarrow Aut(ZH)$ by $\rho(t) = \rho_t$. Then, $\rho$ is a homomorphism. Thus, we can construct the semidirect product $G \stackrel{\rho}{\ltimes} ZH$, where the operation is given by:

\begin{center}
$(t,\alpha)\stackrel{\rho}{\cdot} (t',\alpha') = (tt',\alpha d_t \alpha' d_{t^{-1}})$ \hspace{0.3cm} $(t,t' \in G,\,\, \alpha,\alpha'\in ZH)$
\end{center}

We can now state and prove the main theorem of this paper.

\begin{teo}
\label{teo3} Let $G$ be a finite nonabelian group such that $M(G,2)$ has a nontrivial half-automorphism. Then,

\begin{center}
$\mathit{Half}(M(G,2)) \cong C_2 \times (Aut(G)\stackrel{\sigma}{\ltimes} DH)$

\end{center}

\noindent{}where $DH \cong (G \stackrel{\rho}{\ltimes} ZH))/W$, $ZH \cong C_2^n$ and $W = \{(z,d_z)\,\,|\,\,z\in \mathcal{Z}(G)\} \cong \mathcal{Z}(G)$. Furthermore, $|\mathit{Half}(M(G,2))| = 2 \, |Aut(G)| \, |G|\, |H| = 2^{n+1} \, |Aut(G)|\, |\mathcal{I}(G)|$.
\end{teo}
\begin{proof}
Define $\psi: G \stackrel{\rho}{\ltimes} ZH \rightarrow  DH$ by $\psi((t,\alpha)) = \alpha d_t$. It is clear that $\psi$ is surjective. Since $\alpha d_t \alpha' d_{t'} = \alpha d_t \alpha'd_{t^{-1}} d_{tt'}$, for all $\alpha,\alpha' \in ZH$ and $t,t'\in G$, it follows that $\psi$ is a homomorphism. 

Suppose that $\psi((t,d_z\varphi)) = I$, for some $t\in G$, $z\in \mathcal{Z}(G)$ and $\varphi \in H$. Then, $I = d_z\varphi d_t = \varphi d_{zt}$. Since $H\cap D = \{I\}$, we have $\varphi = I$ and $zt = 1$. Thus, $Ker(\psi) = \{(z,d_z)\,|\, z\in \mathcal{Z}(G)\}$. Therefore, $DH \cong (G \stackrel{\rho}{\ltimes} ZH))/W$, where $W = Ker(\psi)$.

For $(z,d_z),(z',d_{z'})\in W$, we have that $(z,d_z)\stackrel{\rho}{\cdot} (z',d_{z'}) = (zz',d_{zz'})$. Hence, $W \cong \mathcal{Z}(G)$.

The rest of the claim follows from \eqref{eq2} and Propositions \ref{prop9} and \ref{prop11}.
\end{proof}

In the following, some examples of Chein loops that have nontrivial half-automorphisms are provided. The generators of $H$ in these examples were presented in \cite[Examples 8 and 9]{GG13}. In each example, the order of $H$ was obtained by using GAP \cite{gap} computing with the LOOPS package \cite{NV1}.

\begin{exem} Let $Q_8 = \{1,i,j,k,-1,-i,-j,-k\}$ be the quaternion group, where the operation is given by the rules: $i^2 = j^2 = k^2 = -1$, $(-1)^2 = 1$, $ij = k = -ji$, $jk = i = -kj$, and $ki = j = -ik$. The group $H$ has order $8$ and is generated by the nontrivial half-automorphisms

\begin{center}
$\varphi_x(gu^l) = \left\{\begin{array}{rl}
g, & \textrm{if } l = 0,\\
gu, & \textrm{if } l = 1\, \textrm{ and } g \not \in \{x,-x\}, \\
(-g)u, & \textrm{if } l = 1\, \textrm{ and } g \in \{x,-x\},
\end{array}\right.$
\end{center}

\noindent{}where $x\in \{i,j,k\}$. Since $|Aut(Q_8)| = 24$, we have $|\mathit{Half}(M(Q_8,2))|  = 3072$, $|\mathit{Half}_T(M(Q_8,2))| = 384$ and $|Aut(M(Q_8,2))| = 192$. Then, $M(Q_8,2)$ has $2688$ nontrivial half-automorphisms.

\end{exem}

\begin{exem} Let $G = C_4 \ltimes C_3$ be the nonabelian semidirect product between the cyclic groups of orders $4$ and $3$. In this case, $H = \{I,\varphi\}$, where $\varphi$ is the nontrivial half-automorphism defined by

\begin{center}
$\varphi(gu^i) = \left\{\begin{array}{rl}
g, & \textrm{if } i = 0,\\
gu, & \textrm{if } i = 1\,\textrm{ and } o(g)\not = 4, \\
(g^{-1})u, & \textrm{if } i = 1\,\textrm{ and } o(g) = 4.
\end{array}\right.$
\end{center}

Since $|Aut(G)| = 12$, we have $|\mathit{Half}(M(G,2))|  = 576$, $|\mathit{Half}_T(M(G,2))| = 288$ and \\$|Aut(M(G,2))| = 144$. Then, $M(G,2)$ has $288$ nontrivial half-automorphisms.
\end{exem}

\section*{Acknowledgments}
Some calculations in this work have been made by using the LOOPS package \cite{NV1} for GAP \cite{gap}.

\end{document}